\documentclass{amsart}
\usepackage{latexsym, url}
\usepackage{hyperref}
\usepackage[utf8]{inputenc}
\usepackage{amsthm}
\usepackage{amsmath}
\usepackage{amsfonts}
\usepackage{amssymb}
\usepackage[final]{showkeys}    %Remove the labels
\usepackage[dvips]{graphicx}
\usepackage{xypic}
\input xy
\xyoption{all}

\newtheorem{theo}{Theorem}[section]
\newtheorem{thm}[theo]{Theorem}

\newtheorem{fact}[theo]{Fact}
\newtheorem{nota}[theo]{Notation}

\newtheorem{propo}[theo]{Proposition}
\newtheorem{defi}[theo]{Definition}
\newtheorem{defin}[theo]{Definition}

\newtheorem{remark}[theo]{Remark}

\newcommand\lsnk{LS(\ck)}

\newcommand\sless{\triangleleft}
\newcommand\sgreat{\triangleright}
\newcommand\sgeq{\trianglerighteq}

\newcommand\Mod{\operatorname{Mod}}

\newcommand\Homk{\operatorname{Hom}_\ck}

\newcommand\Set{\operatorname{\bf Set}}

\newcommand\Str{\operatorname{\bf Str}}

\newcommand\Ab{\operatorname{\bf Ab}}
\newcommand\FrAb{\operatorname{\bf FrAb}}

\newcommand\dom{\operatorname{dom}}

\newcommand\ca{\mathcal {A}}

\newcommand\calf{\mathfrak {F}}
\newcommand\cu{\mathcal {U}}

\newcommand\ck{\mathcal {K}}
\newcommand\cl{\mathcal {L}}
\newcommand\cm{\mathcal {M}}

\newcommand\cp{\mathfrak {P}}

\newcommand{\pres}[2]{\operatorname{\bf Pres}_{#1}(#2)}
\newcommand{\Red}[3]{\operatorname{\bf Red}_{#1,#2}(#3)}
\newcommand{\Mono}[1]{\operatorname{\bf Mono}(#1)}

\newcommand{\LS}{\operatorname{LS}}
\newcommand{\Ll}{\mathbb{L}}

 \newbox\noforkbox \newdimen\forklinewidth
\forklinewidth=0.3pt \setbox0\hbox{$\textstyle\smile$}
\setbox1\hbox to \wd0{\hfil\vrule width \forklinewidth depth-2pt
 height 10pt \hfil}
\wd1=0 cm \setbox\noforkbox\hbox{\lower 2pt\box1\lower
2pt\box0\relax}

\newcommand{\bK}{\mathbb{K}}

\newcommand{\cP}{\mathcal{P}}
\newcommand{\bL}{\mathbb{L}}

\newcommand{\comment}[1]{}

\title[Tameness, powerful images, and large cardinals]
      {Tameness, powerful images, and large cardinals}
\date{\today\\
AMS 2010 Subject Classification: Primary: 03E55, 18C35. Secondary: 03E75, 03C20, 03C48, 03C75.
}
\keywords{weakly compact cardinals, accessible categories, abstract elementary classes, Galois types, locality}

\author[W. Boney]{Will Boney}
\urladdr{http://wboney.wp.txstate.edu}
\address{Mathematics Department \\ Texas State University \\ San Marcos, Texas, USA}

\author[M. Lieberman]{Michael Lieberman}
\urladdr{http://www.math.muni.cz/\textasciitilde lieberman/}
\address{Institute of Mathematics, Faculty of Mechanical Engineering, Brno University of Technology, Brno, Czech Republic}
\address{Department of Mathematics and Statistics, Faculty of Science, Masaryk University, Brno, Czech Republic}

\begin{document}

\begin{abstract} We provide comprehensive, level-by-level characterizations of large cardinals, in the range from weakly compact to strongly compact, by closure properties of powerful images of accessible functors.  In the process, we show that these properties are also equivalent to various forms of tameness for abstract elementary classes.  This systematizes and extends results of \cite{boun}, \cite{BT-R}, \cite{lcatth}, and \cite{LRclass}.
\end{abstract} 

\maketitle

%\tableofcontents

\section{Introduction}

Recent years have seen the rapid development of a literature concerning the implications of set-theoretic assumptions---and, in particular, large cardinal axioms---for the structure of certain useful classes of categories.  This is particularly true for \emph{accessible categories}, a category-theoretic rendering of the classes of models axiomatizable in infinitary logics (we recall all of the necessary terminology in Section~\ref{prelims} below).  In fact, following \cite{BT-R}, \cite{boun}, and \cite{lcatth}, we will exhibit equivalent characterizations of large cardinals in terms of, on the one hand, clean and practical closure conditions on the images of well-behaved functors between accessible categories, and, on the other hand, good behavior (specifically, \emph{tameness}) of types in abstract model theory.

While the extent to which the structure of accessible categories depends on set theory has come as something of a surprise, it has always been clear that they occupy a complicated place between pure, i.e. abstract, category theory and ensemblist mathematics. Although there is no assumption of, or recourse to, underlying sets of objects in the definitions, constructions, and basic theory of accessible categories, infinite cardinals abound, both as indices of accessibility (``Let $\ck$ be a $\kappa$-accessible category\dots") and as sizes (or, rather, \emph{presentability ranks}) of objects.  One might hope that these cardinals could be treated simply as an ordered family of indices, avoiding their more delicate arithmetical and combinatorial properties, but this has never been a possibility: per the foundational text of Makkai and Par\'e, for every $\mu$ and $\lambda$, every $\lambda$-accessible category is $\mu$-accessible just in case $\mu$ is \emph{sharply larger than} $\lambda$, \cite[2.3.1]{makkai-pare}, a relation that reduces to cardinal arithmetic (\cite[2.5]{internal-sizes-v2}, and Remark~\ref{remsharpclosed} below)\footnote{We note that in \emph{locally presentable categories}---special accessible categories with arbitrary (co)limits, corresponding to models of infinitary \emph{limit theories}---such sensitivities tend to disappear, barring extreme cases along the lines of Vop\v enka's Principle, cf. \cite[\S 6.B]{adamek-rosicky}}.  That is, the \emph{accessibility spectrum} of a general accessible category is sensitive to cardinal arithmetic, and simplifies considerably under, e.g. instances of SCH (\cite[2.5, 2.7]{internal-sizes-v2}).  Similar phenomena are observed in connection with the \emph{existence spectrum} of an accessible category $\ck$, i.e. those cardinals $\kappa$ such that it contains an object of presentability rank $\kappa$.  The fact that any sufficiently large object has presentability rank a successor cardinal follows, for example, from GCH (\cite[2.3(5)]{rosbek}) or, again, instances of SCH (\cite[3.11]{internal-sizes-v2}, \cite[5.5]{lrvfilt}).  As SCH holds above a sufficiently strongly compact cardinal (see \cite[20.8]{jechbook}) results of this form also arise as consequences of a large cardinal assumption.

Following, and ultimately systematizing, results of \cite{makkai-pare}, \cite{BT-R}, and \cite{lcatth}, we here concern ourselves with a more straightforwardly category-theoretic family of consequences of large cardinals: the closure properties of the images of \emph{accessible functors}, i.e functors between accessible categories that preserve sufficiently directed colimits.  The historical question motivating this line of inquiry is refreshingly simple: consider the category $\Ab$ consisting of abelian groups and group homomorphisms, and the full subcategory consisting of free abelian groups, $\FrAb$.  While $\Ab$ is finitely accessible (it has all directed colimits---that is, direct limits---and any abelian group is a directed colimit of finitely presented ones), things are much more complicated when it comes to $\FrAb$.  Under V=L, $\FrAb$ is not closed under $\kappa$-directed colimits in $\Ab$ for any uncountable regular $\kappa$: one can always construct a group $G$ of size $\kappa$ that is indecomposable---so, in particular, non-free---all of whose $<\kappa$-generated subgroups are free (\cite{emindec}, building on \cite{shegrp}).  Assuming a strongly compact cardinal $\kappa$, on the other hand, this problem disappears: per \cite[3.10]{emalmostfree}, any abelian group whose $<\kappa$-generated subgroups are free is free, and $\FrAb$ is closed under $\kappa$-directed colimits in $\Ab$.  That is, $\FrAb$ is \emph{$\kappa$-accessibly embedded} in $\Ab$.   To obtain this phenomenon at small cardinals (although still using large cardinals), see \cite{magidor-shelah}.

As a first step toward generalizing this result, note that $\FrAb$ is precisely the image of the free functor $F:\Set\to\Ab$, which takes a set to the free abelian group on its elements.  In fact, $\FrAb$ is the \emph{powerful image} of this functor, i.e. the closure of the image of $F$ under subobjects in $\Ab$: since any subgroup of a free abelian group is free abelian, $\FrAb$ is already closed in this sense.  The functor $F$ is finitely accessible---it preserves directed colimits---and, as we have just seen, closure of its powerful image under colimits in the codomain $\Ab$ is highly sensitive to set theory.  One might then be inclined to ask, for example, whether, and to what extent, closure properties of the powerful images of general accessible functors are similarly conditioned by the existence of large cardinals or other set-theoretic hypotheses.  The cornerstone result in this area is \cite[5.5.1]{makkai-pare}: if $\kappa$ is strongly compact, and $F:\ck\to\cl$ is an accessible functor below $\kappa$, then the powerful image of $F$ is $\kappa$-accessibly embedded in $\cl$.\footnote{In fact, the standard gloss of this result is that ``the powerful image is $\kappa$-accessible,'' or perhaps ``the powerful image is $\kappa$-accessible and $\kappa$-accessibly embedded'' (\cite{BT-R}), but as we highlight in Theorems~\ref{thmcptclosed} and ~\ref{thmcptclosednok}, (pre)accessibility of the powerful image holds without any assumption beyond ZFC.  We therefore zero in on the closure condition on the powerful image, which genuinely depends on the large cardinal type of $\kappa$, rather than the structure of the powerful image as a category in its own right.}  Careful analysis of the proof of this theorem in \cite{BT-R} led not only to the realization that $\kappa$ need only be \emph{almost} strongly compact, but also a clearer parsing of the connection between the properties of $\kappa$ as a large cardinal and the closure of the powerful image under colimits of a given shape.  This allows one to link weaker compactness notions with weaker closure conditions---see \cite{lcatth}, and Section~\ref{seccomppow} below.

This is, however, just the first link in a chain of implications connecting large cardinals, category theory, and abstract model theory.  The last of the three enters the picture via \cite{boneylarge} and \cite{LRclass}.  The centerpiece result of the former is that if $\kappa$ is a strongly compact cardinal, the Galois types in any \emph{abstract elementary class} (or \emph{AEC}) essentially below $\kappa$ are $<\kappa$-tame; that is, they are completely determined by their restrictions to $<\kappa$-sized submodels of their domains \cite[4.5]{boneylarge}.  The same result was subsequently derived in \cite{LRclass}, but by different means---one can characterize equivalence of Galois types via the powerful image of an accessible functor, in which case $<\kappa$-tameness corresponds precisely to $\kappa$-accessible embeddability of this powerful image.  By \cite[5.5.1]{makkai-pare}, as we have seen, $\kappa$-accessible embeddability, too, follows from strong compactness of $\kappa$.

This chain of implications can be tightened to form an equivalence---as noted above, \cite[3.2]{BT-R} shows that any accessible functor below an almost strongly compact cardinal $\kappa$ has powerful image $\kappa$-accessibly embedded in its codomain.  Moreover, we have seen that this in turn is enough to give $<\kappa$-tameness of AECs below $\kappa$.  The loop is closed in \cite{boun} (building on \cite{sh932}), where the authors build a specific AEC $\ck$ that, subject to certain technical conditions, allows one to infer an instance of compactness from an instance of tameness of $\ck$.  So we are left with an equivalence:
\begin{thm}[{\cite[Corollary 4.14]{boun}}]\label{buthm}
	Let $\kappa$ be an infinite cardinal with $\mu^\omega<\kappa$ for all $\mu<\kappa$.  The following are equivalent:
	\begin{enumerate}
		\item $\kappa$ is almost strongly compact.
		\item The powerful image of any accessible functor $F:\ck\to\cl$ below $\kappa$ is $\kappa$-accessibly embedded in $\cl$.
		\item Every AEC essentially $\ck$ below $\kappa$ is $<\kappa$-tame.
	\end{enumerate}
\end{thm}

There is a great deal more to be said, however.  In conjunction with \cite{boneylarge}, \cite{boun} gives a much broader and subtler array of equivalences between gradations of compactness and gradations of tameness.  One can also fine-tune the argument of \cite{BT-R} to match these finer gradations.  As noted in \cite{lcatth}, \emph{almost measurability} of a cardinal $\kappa$ implies the closure of powerful images under colimits of $\kappa$-chains, which implies $<\kappa$-locality of types in AECs below $\kappa$.  Again, subject to technical assumptions, \cite{boun} allows us to infer almost measurability from locality, yielding another equivalence along similar lines.  Here, we focus on two different gradings of compactness (along with some parameters):
\begin{itemize}
	\item \emph{$(\delta,\theta)$-compactness}, a relatively classical logical compactness property for $\theta$-sized theories in $\Ll_{\delta, \delta}$ (see Definition \ref{defparamcpt}); and
	\item \emph{$(\delta, \theta)$-compactness for ultrafilters}, which amounts to the existence of certain $\delta$-complete ultrafilters (see Definition \ref{sc-u-def}).
\end{itemize}
While the category-theoretic implications of these principles are different, and we largely treat them in parallel, the two notions correspond, crucially, under a cardinal arithmetic assumption that we must make in proving our main equivalence result along the lines of Theorem~\ref{buthm}.  So ultimately this is a distinction with very little difference.

We give a brief review of the terminology involved in Section~\ref{prelims}, including the precise definitions of the above-mentioned variations on $(\delta,\theta)$-compactness.  Section~\ref{seccomppow} derives corresponding closure conditions for powerful images of accessible functors from $(\delta,\theta)$-compactness (see Theorems~\ref{thmcptclosed} and \ref{thmcptclosednok}).  In Section~\ref{sectame} we derive a suitable form of tameness from the stronger of the two closure conditions, Theorem~\ref{thmclosedtame}.   In Section~\ref{secequiv}, we bring the results of \cite{boun} into play, closing the loop, and obtaining the promised level-by-level equivalence refining Theorem~\ref{buthm}.

Although we have made an effort to keep the paper as self-contained as possible,  basic familiarity with accessible categories, see e.g. \cite{makkai-pare} and \cite{adamek-rosicky}; with AECs and tameness, see e.g. \cite{baldwinbk}; with links between large cardinals and abstract model theory, see \cite{boun} and \cite{boneylarge}. The papers \cite{BT-R}, \cite{boun}, and \cite{lcatth} provide useful context on the broader project of producing model- and category-theoretic characterisations of large cardinals.

We owe a profound debt to the anonymous referee, whose painstaking reading and commentary led to substantial improvements in the exposition and results.

\section{Preliminaries}\label{prelims}

The large cardinals that we consider here are all related to the compactness of various infinitary logics, and provide a grading of large cardinals in the range from weakly compact to strongly compact cardinals.  We note that there has been an incredible proliferation of weak forms of compactness.  The first compactness notion we consider is the logical formulations, as these are easily motivated, correspond clearly to category-theoretic closure conditions, and tend to overlap nicely---at least on a global level---with the other existing formulations (see \cite{hayut} for more on the local behavior).  Recall that $\bL_{\kappa, \lambda}$ refers to the logic that allows conjunction and disjunction of $<\kappa$-many formula and quantification over $<\lambda$-many elements of the universe.  A theory $T \subset \bL_{\kappa, \lambda}$ is called \emph{$<\kappa$-satisfiable} if every $<\kappa$-sized subset has a model.

As motivation for the definition that follows, we recall that finitary first order logic is compact: for any theory $T$ in a finitary language $\bL_{\omega,\omega}$, if $T$ is finitely satisfiable---all of its finite subsets are satisfiable---then $T$ is satisfiable.  Whether this is true in the infinitary case depends on large cardinals: for an uncountable cardinal $\kappa$, any $<\kappa$-satisfiable theory $T$ in an infinitary language $\bL_{\kappa\kappa}$ is satisfiable just in case $\kappa$ is a \emph{strongly compact cardinal}.  Weakly compact cardinals are defined similarly, but with a restriction on the size of the theories involved---$<\kappa$-satisfiability implies satisfiability only for theories in $\bL_{\kappa\kappa}$ of cardinality $\kappa$.  In the latter definition, it is clear that there are, morally speaking, three parameters to play with, namely the infinitary character of the language, the degree of partial satisfiability, and the cardinality of the theories for which the conclusion holds.  This leads us to define:

\begin{defi}\label{defparamcpt}
Let $\kappa$ be an inaccessible cardinal.  We say that $\kappa$ is \emph{$(\delta,\theta)$-compact}, $\delta\leq\kappa\leq\theta$, if every $<\kappa$-satisfiable theory of size $\theta$ in $\bL_{\delta, \delta}$ is satisfiable.  We say $\kappa$ is $(\delta,<\theta)$-compact if it is $(\delta,\theta')$-compact for all $\theta'<\theta$, and $(\delta,\infty)$-compact if it is $(\delta,\theta)$-compact for all $\theta$.
\end{defi}

We note that this encompasses a number of standard notions as special cases:

\begin{enumerate} 
	\item $\kappa$ is strongly compact if and only if it is $(\kappa,\infty)$-compact.
	\item $\kappa$ is almost strongly compact if and only if it is $(\delta, \infty)$-compact for every $\delta < \kappa$.
	\item $\kappa$ is weakly compact if and only if it is $(\kappa,\kappa)$-compact.
\end{enumerate}

\begin{remark}
\begin{enumerate}
	\item Although we have defined compactness in terms of $\bL_{\delta, \delta}$, we would get the same strength if it were defined in terms of $\bL_{\delta,\omega}$.
	\item The `almost' adjective was introduced in \cite{bagaria-magidor} and indicates that the strong compactness is only guaranteed to exist cofinally in $\kappa$.
\end{enumerate}
\end{remark}

We also make use of the ultrafilter formulation of compactness.  Our motivation for this is primarily practical: the parameterized, logical compactness characterizations yield an additional hypothesis (for instance, the requirement that the diagrams in Theorem \ref{thmcptclosed}.(2) are made up of $\kappa$-presentables) which was not present in previous versions and that only seemed to be removed by cardinal arithmetic assumptions (typically $\theta = \theta^{<\kappa}$).  However, the ultrafilter formulation removes these extra hypotheses naturally, and the previous use of cardinal arithmetic is explained by its use in Remark \ref{c-u-sc-rem}.

\begin{defin}\label{sc-u-def}
$\kappa$ is \emph{$(\delta, \theta)$-compact for ultrafilters} iff there is a fine, $\delta$-complete ultrafilter on $\cP_\kappa \theta$, where $\cP_\kappa\theta$ is the set of all subsets of $\theta$ of cardinality less than $\kappa$.
\end{defin}

Recall that an ultrafilter $\cu$ on $\cP_\kappa\theta$ is said to \emph{fine} if for any $\alpha\in\lambda$, the set
$$[\alpha]:=\{X\in\cP_\kappa\theta\,|\,\alpha\in X\}$$
is in $\cu$.  If $\cu$ is also $\kappa$-complete, this implies that, for every $s \in \cP_\kappa \theta$, the set $\{X \in \cP_\kappa\theta\,|\,s\subset X\}$ is in $\cu$.  On a global level, $(\delta, \infty)$-compactness and $(\delta, \infty)$-compactness for ultrafilters coincide.  However, on a local level, the notion of compactness for ultrafilters is stronger.  \cite[Theorem 3.5 and Proposition 3.9]{boneymodlcs} emphasized the connection between the existence of fine, highly complete ultrafilters on one hand and compactness based on \emph{filtrations} of theories (as opposed to their cardinality) on the other.  In this case, the argument gives the following characterization of compactness for ultrafilters.

\begin{fact}\label{log-comp-c-c-fact}
$\kappa$ is $(\delta, \theta)$-compact for ultrafilters iff if a theory $T$ in $\Ll_{\delta, \delta}(\tau)$ can be written as an increasing union $T = \cup_{s \in \cP_\kappa \theta} T_s$ such that each $T_s$ is satisfiable, then $T$ is satisfiable.
\end{fact}

This implies the cardinality-based satisfiability criteria of a $(\delta, \theta)$-compact cardinal, but is much more powerful; in particular, it also applies to definable class theories.  We will use a version of this statement more closely adapted our purposes:

\begin{propo}\label{categorified-prop}
Let $\kappa$ be $(\delta, \theta)$-compact for ultrafilters and $I$ a $\kappa$-directed poset, $|I|\leq\theta$.  If $T$ is an $\Ll_{\delta, \delta}(\tau)$-theory that can be written as an increasing union $T = \cup_{i \in I} T_i$ indexed by $I$ such that each $T_i$ is satisfiable, then $T$ is satisfiable.
\end{propo}

\begin{proof}
The proof follows the proof of Fact \ref{log-comp-c-c-fact} in \cite{boneymodlcs}, so we only sketch the proof here.  Given the ultrafilter $U$ from Definition \ref{sc-u-def}, we can form the resulting elementary embedding $j:V \to \mathcal{M}$ with $\delta \leq \text{crit } j \leq \kappa$ and $j(\kappa) > \theta$ such that any $\theta$-sized subset of $\mathcal{M}$ is covered by a set in $\mathcal{M}$ of size $<j(\kappa)$.

Using this property, we can find, in $\mathcal{M}$, a set $A \subset j(I)$ of size $<j(\kappa)$ that contains $j``I$.  Since $j(I)$ is $j(\kappa)$-directed in $\mathcal{M}$, there is $i_* \in j(I)$ that is above all of $A$.  Setting 
$$\{T^*_i : i \in j(I)\} = j\left(\{T_i : i \in I\}\right)$$
we have that, for all $i \in I$,
$$j(T_i) = T^*_{j(i)} \subset T^*_{i_*}$$
This means that $j``T \subset T^*_{i_*}$.  In $\mathcal{M}$, $T^*_{i_*}$ is satisfiable.  Since $\delta \leq \text{crit j}$, this is absolute between $V$ and $\mathcal{M}$.  Thus, $j``T$ is satisfiable.  This differs from $T$ only by a renaming according to $j$, so $T$ is satisfiable (see \cite{boneymodlcs} for more details on this style of argument).
\end{proof}

In the other direction, one can write down an $\mathbb{L}_{\delta, \delta}$-theory $T$ of size $\theta^{<\kappa}$ that is $<\kappa$-satisfiable and is satisfiable  precisely when there is a $\delta$-complete, fine ultrafilter on $\cP_\kappa \theta$ (this is folklore; \cite[Theorem 3.5(1)$\to$(3)]{boneymodlcs} records a more general argument).  Combining this, we arrive at the following summary:

\begin{remark}\label{c-u-sc-rem}
If $\kappa$ is $(\delta, \theta)$-compact for ultrafilters, then it is $(\delta, \theta)$-compact.  If $\theta = \theta^{<\kappa}$, then the converse holds as well.
\end{remark}

  See Hayut \cite{hayut} for more on the local implications between characterizations of compactness.
%\begin{remark}
%As can be seen in standard texts (e.g., \cite{kanamori}), the logical global versions (`logically weakly compact,' `logically strong compact,' etc.) agree with the standard global versions (`weakly compact,' `strong compact,' etc.).  On a level-by-level basis (where the relevant definitions can be found in \cite[Definition 2.1]{boun}, things separate a bit more. Hayut \cite{h-partial-strong} studies this behavior when $\delta= \kappa$, and the same arguments work in $\delta<\kappa$ case.  Consider the following three notions:
%\begin{equation} \label{cto-set} \tag{$(*)_{\delta, \kappa, \lambda}$}
%\text{$\kappa$ is logically $(\delta,\lambda)$ compact}
%\end{equation}
%\begin{equation} \label{cto-set} \tag{$(**)_{\delta, \kappa, \lambda}$}
%\text{any $\kappa$-complete filter on $\lambda$ can be extended to a $\delta$-complete ultrafilter}
%\end{equation}
%\begin{equation} \label{cto-set} \tag{$(***)_{\delta, \kappa, \lambda}$}
%\text{there is a $\delta$-complete, fine ultrafilter on $\cP_\kappa, \lambda$}
%\end{equation}
%Then, if $\delta<\kappa \leq \lambda=\lambda^{<\kappa}$, we have\footnote{See \cite{h-partial-strong}, especially  Theorem 3 there; Hayut uses `$\bL_{\kappa, \kappa}$-compactness for languages of size $2^\lambda$' to refer to $(**)_{\kappa, \kappa, \lambda}$.}
%$$(***)_{\delta, \kappa, 2^\lambda} \implies (**)_{\delta, \kappa, \lambda} \iff (*)_{\delta, \kappa, 2^\lambda} \implies (***)_{\delta, \kappa, \lambda}$$
%We make use of the logical characterizations of these cardinals, so prefer this definition.
%\end{remark}

We now recall the following terminology related to accessible categories, and refer readers to \cite{adamek-rosicky} and \cite{makkai-pare} for further details:

\begin{defin} Let $\lambda$ be a regular cardinal.
	\begin{enumerate} 
		\item A diagram $D:I\to\ck$ in a category $\ck$ is \emph{$\lambda$-directed} if the underlying poset $I$ is $\lambda$-directed, i.e. any $J\subset I$ with $|J|<\lambda$ has an upper bound in $I$.
		\item An object $M$ in a category $\ck$ is \emph{$\lambda$-presentable} if the associated hom-functor
		$$\Homk(M,-):\ck\to\Set$$
		preserves $\lambda$-directed colimits.
		\item A category $\ck$ is $\lambda$-preaccessible if it contains a set $\ca$ of $\lambda$-presentable objects such that any object of $\ck$ is a $\lambda$-directed colimit of objects in $\ca$.
		\item A category $\ck$ is $\lambda$-accessible if it is $\lambda$-preaccessible and has all $\lambda$-directed colimits.  We say $\ck$ is accessible if it is $\lambda$-accessible for some $\lambda$.
	\end{enumerate}
\end{defin}

Incidentally, we will often need the following additional notion related to colimits:

\begin{defin}\label{defsmall}
	Let $\theta$ be an infinite cardinal.  We say that a diagram $D:I\to\ck$ is \emph{$\theta$-small} if $|I|<\theta$.  A colimit is $\theta$-small if its underlying diagram is $\theta$-small.
\end{defin}

\begin{defin}\label{defpresrank}
	\emph{In an accessible category $\ck$, every object $M$ is $\lambda$-presentable for some $\lambda$.  We define the \emph{presentability rank} of $M$ to be the least such $\lambda$.}
\end{defin}

\begin{remark}\label{remsharpclosed}
	\emph{Recall that accessibility (and preaccessibility) do not pass upward in an entirely straightforward fashion: while a $\lambda$-accessible category may be accessible in all regular cardinals $\mu\geq\lambda$ (i.e. it may be \emph{well accessible}, \cite[2.1]{rosbek}) in nice special cases, it will at least be accessible in a proper class of such $\mu$.  In particular, any $\lambda$-accessible category is $\mu$-accessible, $\mu>\lambda$ regular, if $\mu$ is \emph{sharply greater} than $\lambda$, denoted $\mu\sgreat\lambda$.  This sharp inequality relation, defined in \cite[2.3.1]{makkai-pare}, admits a number of equivalent characterizations (see \cite[2.11]{adamek-rosicky}), e.g. $\mu\sgreat\lambda$ if and only if any subset of cardinality $\mu$ in a $\lambda$-directed poset can be completed to a $\lambda$-directed set of cardinality $\mu$.  We note that $\sless$ is strictly stronger than $<$: by \cite[2.13(8)]{adamek-rosicky}, for example, it is not the case that $\aleph_1\sless\aleph_{\omega+1}$.}
\end{remark}

We note that this relation admits a clean characterization in terms of cardinal arithmetic.  Recall:

\begin{defin}\label{defclosed}
	Let $\lambda$ be a cardinal.  We say that a cardinal $\mu$ is \emph{$\lambda$-closed} if $\theta^{<\lambda}<\mu$ for all $\theta<\mu$.
\end{defin}

The set-theoretic content of the following seems to have been in the folklore for some time, cf. \cite{abmag}: in the form below, it is \cite[2.5]{internal-sizes-v2}.

\begin{fact}\label{factsharp}
	Let $\mu$ and $\lambda$ be regular cardinals, $\mu>\lambda$.  If $\mu$ is $\lambda$-closed, $\mu\sgreat\lambda$; when $\mu>2^{<\lambda}$, the converse holds.
\end{fact}

\begin{defin}\label{functsdef}
	Let $\lambda$ be a regular cardinal.
	\begin{enumerate}
		\item We say that a functor $F:\ck\to\cl$ is $\lambda$-accessible if $\ck$ and $\cl$ are $\lambda$-accessible, and $F$ preserves $\lambda$-directed colimits.
		\item\label{accemb-functsdef} Given a subcategory $\ck$ of a category $\cl$, we say $\ck$ is $\lambda$-accessibly embedded in $\cl$ if it is full and $\ck$ is closed under $\lambda$-directed colimits in $\cl$.
	\end{enumerate}
\end{defin}

The central mission of the current paper, following \cite{BT-R} and \cite{lcatth}, is to determine how close the \emph{powerful image} of an accessible functor $F:\ck\to\cl$ comes to being accessibly embedded in $\cl$, provided we work below a large cardinal $\kappa$.  Recall:

\begin{defin}\label{defimgs}
Let $F:\ck\to\cl$ be a functor.
\begin{enumerate}
	\item The \emph{full image of $F$} is the full subcategory of $\cl$ on objects $FA$, $A\in\ck$.  For brevity, we denote the full image of $F$ by $\calf(F)$.
	\item\label{powim-def} The \emph{powerful image of $F$}, denoted $\cp(F)$, is the closure of $\calf(F)$ under $\cl$-subobjects\footnote{We note that it is sometimes useful to take the \emph{$\lambda$-pure powerful image} of $F$, $\cp_\lambda(F)$, namely the closure of $\calf(F)$ under \emph{$\lambda$-pure} (mono)morphisms. Indeed, \cite{BT-R} focuses on this notion, and treats the powerful image as secondary (e.g. \cite[3.4]{BT-R}).  We take the opposite---less technically demanding---approach, indicating in Remark~\ref{rmklpure} how our argument can be easily modified to yield analogous results for $\lambda$-pure powerful images.}: for any $M$ in $\cl$, if there is an $\cl$-monomorphism $M\to N\in\calf(F)$, $M$ is in $\cp(F)$.
\end{enumerate}	
\end{defin}

As mentioned above, we will be concerned with accessible functors \emph{below} particular large cardinals, a notion which we now make precise through a series of definitions: 

\begin{remark}\label{rmkmu}
	\emph{We recall an important parameter, $\mu_{\ck,\lambda}$, associated with a $\lambda$-accessible category $\ck$ (\cite[3.1]{BT-R}).  We proceed as follows:
	\begin{enumerate}
		\item For any cardinal $\beta$ and regular cardinal $\lambda$, define $\gamma_{\beta,\lambda}$ to be the least cardinal $\gamma$ with $\gamma\geq\beta$ and $\gamma\sgeq\lambda$.  Define $\mu_{\beta,\lambda}$, in turn, by
		$$\mu_{\beta,\lambda}=(\gamma_{\beta,\lambda}^{<\gamma_{\beta,\lambda}})^+.$$
		\item If $\ck$ is a $\lambda$-accessible category, let $\mu_{\ck,\lambda}$ be $\mu_{\beta,\lambda}$, where $\beta$ is the cardinality of $\pres{\lambda}{\ck}$, i.e. a full subcategory on a set of representatives of the isomorphism classes of $\lambda$-presentable objects.
		\item As we have seen, any accessible category $\ck$ is $\lambda$-accessible for many $\lambda$: we write $\mu_\ck$ to mean $\mu_{\ck,\lambda}$, where $\lambda$ is the least such cardinal.
	\end{enumerate}
	By design, $\mu_{\ck,\lambda}\sgeq\gamma_{\ck,\lambda}\sgeq\lambda$.  As we will see in Section~\ref{seccomppow}, this parameter gives a bound on the linguistic resources needed to describe $\ck$ as a well-behaved class of structures.}
\end{remark}

We define a related parameter associated with each accessible functor:

\begin{defin}\label{defmufunct}
Let $F:\ck\to\cl$ be an accessible functor, and let $\lambda$ be the least cardinal such that $F$ is $\lambda$-accessible.  Define $\mu_F$ to be the least regular cardinal $\mu\geq\mu_{\cl,\lambda}$ such that $F$ preserves $\mu$-presentable objects; that is, if $M$ is $\mu$-presentable in $\ck$, $F(M)$ is $\mu$-presentable in $\cl$.
\end{defin}

\begin{defin}\label{deffunctbelow}
	Let $\kappa$ be a cardinal.  We say that an accessible functor $F:\ck\to\cl$ is \emph{below} $\kappa$ if $\mu_F<\kappa$.  We say that a particular accessible category $\ck$ is \emph{below} $\kappa$ if $\mu_\ck<\kappa$.
\end{defin}

We also review a few of the ideas we need in connection with abstract elementary classes (AECs).  First defined in \cite{shelahaecs}, AECs are a semantic (or, if you like, category-theoretic) abstraction of well-behaved nonelementary classes of models and embeddings, such as those axiomatizable in ${\mathbb L}_{\lambda,\omega}$.  Naturally, they also generalize elementary classes from finitary first-order logic $\bL_{\omega, \omega}$.  Crucially, AECs retain, in the structure of their class of designated embeddings, certain essential properties of these more elementary cousins, particularly those proved without any appeal to compactness: downward L\"{o}wenheim-Skolem, the Tarski-Vaught chain, etc.  For our purposes, it is enough, perhaps, to recall that an AEC $\ck$ has arbitrary directed colimits of $\ck$-embeddings, and that it has an associated L\"owenheim-Skolem number $\lsnk$ such that any object in $\ck$ is a $\lsnk^+$-directed colimit of $\lsnk^+$-presentable objects.  Here the presentability rank of an object $M\in\ck$ is precisely $|UM|^+$, i.e. the successor of the cardinality of the underlying set of $M$.

\begin{nota}\label{notationforgetful}
\emph{Whenever we work in an AEC (or, indeed, any concrete category) $\ck$, we denote by $U$ the functor $U:\ck\to\Set$ that sends each object to its underlying set.} 
\end{nota}

In AECs, the syntactic types familiar from first-order model theory are replaced by \emph{Galois} (or \emph{orbital}) \emph{types}, a notion that makes sense even in a general concrete category.  We adopt the more category-theoretic formulation introduced in \cite{LRclass}, where we identify types in a concrete category $\ck$ with equivalence classes of pairs of the form $(f,a)$, where $f:M\to N$ and $a\in UN$.  In particular, we say that pairs $(f_1,a_1)$ and $(f_2,a_2)$ are equivalent if the pointed span 
$$a_1\in N_1\stackrel{f_1}{\leftarrow} M\stackrel{f_2}{\to} N_2\ni a_2$$
can be extended to a commutative square 
$$  
      \xymatrix{
        N_1 \ar[r]^{g_1} & N \\
        M \ar [u]^{f_1} \ar [r]_{f_2} &
        N_2 \ar[u]_{g_2}
      }
    $$
with $g_1(a_1)=g_2(a_2)$.  Transitivity of this relation is not automatic: one could simply take its transitive closure, but this is somewhat unwieldy from a technical perspective.  If, on the other hand, we assume that $\ck$ has the amalgamation property, i.e. any span $N_1\leftarrow M\to N_2$ can be completed to a commutative square, this relation is already transitive.  For the sake of simplicity, we adopt the latter approach.\footnote{In fact, the arguments of this paper will also go through in the former, more technical, case, with a slight modification to Definition~\ref{deftame}, and a bit of extra bookkeeping.}

The notion of \emph{tameness} of Galois types in a concrete category $\ck$---essentially the requirement that equivalence of of pairs $(f_1,a_1)$ and $(f_2,a_2)$ over any $M$ is determined by restrictions to subobjects of $M$ of some fixed small size---was first isolated in \cite{grovantame}, and has come to play an important role in the development of the classification theory of AECs.  We consider several parameterizations of this notion, again presenting a category-theoretic formulation inspired by \cite[5.1]{LRclass}.

\begin{defi}\label{deftame}
	Let $\ck$ be a concrete category, and let $\kappa\leq\theta$.
	\begin{enumerate}
		\item We say that $\ck$ is \emph{$(<\kappa,\theta)$-tame} if the following holds for every $\theta^+$-presentable object $M$: given any $(f_1,a_1)$ and $(f_2,a_2)$ over $M$, if $(f_1\chi,a_1)$ is equivalent to $(f_2\chi,a_2)$ for all $\chi:X\to M$ with $X$ $\kappa$-presentable, then $(f_1,a_1)$ and $(f_2,a_2)$ are equivalent.
		\item We say $\ck$ is \emph{$<\kappa$-tame} if it is $(<\kappa,\theta)$-tame for all $\theta\geq\kappa$.
	\end{enumerate}
\end{defi}

In light of the correspondence between cardinality and presentability in an AEC $\ck$ (see the comment preceding \ref{notationforgetful} above), these clearly reduce to the standard definitions in that context, cf. \cite[11.6]{baldwinbk}.

\section{Compactness and powerful images}\label{seccomppow}

We are now in a position to state the main result connecting $(\delta,\theta)$-compactness to the closure properties of powerful images.  First, though, we establish the broader context for this result, encompassing \cite{BT-R}, \cite{lcatth}, and, more concretely, the free abelian group example considered in the introduction.

In connection with the latter, we note again that the free functor
$$F:\Set\to\Ab$$
is finitely accessible (i.e. $\omega$-accessible), and that $\cp(F)=\calf(F)=\FrAb$, the full subcategory of $\Ab$ in the free abelian groups.  We have the following:

\begin{remark}\label{rmkfreeab} \emph{Let $F:\Set\to\Ab$ be as above, $\kappa$ an infinite cardinal.
\begin{enumerate}
	\item $\FrAb$ is $\kappa$-preaccessible in $\Ab$, i.e. any $G\in\FrAb$ is a $\kappa$-directed colimit in $\Ab$ of $<\kappa$-presentable objects.  In particular, $G$ is a $\kappa$-directed colimit of its $<\kappa$-presented subgroups, which, being subgroups of a free abelian group, must be free abelian.
	\item If $\kappa$ is strongly compact, $\FrAb$ is closed under $\kappa$-directed colimits in $\Ab$, i.e. it is $\kappa$-accessibly embedded in $\Ab$.  Notice that $\kappa\sgreat\omega$ (this holds, in fact, for any uncountable regular cardinal), so $\Ab$ is $\kappa$-accessible.  By (1), and the fact that $\FrAb$ is $\kappa$-accessibly embedded in $\Ab$, it follows that $\FrAb$ is $\kappa$-accessible.
\end{enumerate}}
\end{remark}

As we noted in the introduction, the closure in Remark~\ref{rmkfreeab}(2) genuinely depends on the large cardinal character of $\kappa$---if, for example, we take the extreme position that $V=L$, we can arrange a counterexample to closure under $\kappa$-directed colimits for any $\kappa$.  The cumulative result of the generalizations of this example---\cite[5.5.1]{makkai-pare}, \cite[3.3]{lcatth}, and particularly the proof of \cite[3.2]{BT-R}---is a clear suggestion that large cardinal properties translate cleanly to closure conditions on powerful images of accessible functors.  We summarize those results here:

\begin{fact}\label{factlcsumm} \emph{
	Let $\kappa$ be an infinite cardinal, $F:\ck\to\cl$ an accessible functor below $\kappa$.
	\begin{enumerate}
	\item Without any large cardinal assumptions whatsoever, the powerful image of $F$, $\cp(F)$, is $\kappa$-preaccessible in $\cl$: any object $M$ in $\cp(F)$ is a $\kappa$-directed colimit of $\kappa$-presentable objects in $\cl$ that lie in $\cp(F)$.
	\item If $\kappa$ is an almost [measurable/strongly compact] cardinal, $\cp(F)$ is closed under [colimits of $\kappa$-chains/$\kappa$-directed colimits] in $\cl$.
	\end{enumerate} }
\end{fact} 

For the remainder of this section, we will use the machinery of the aforementioned papers to determine the closure conditions corresponding to $(\delta,\theta)$-compactness of $\kappa$.  First, we give the version for strong compactness (as in Definition \ref{defparamcpt}).

\begin{theo}\label{thmcptclosed}
	Let $\kappa$ be $(\delta,\theta)$-compact.  For any accessible functor $F:\ck\to\cl$ below $\delta$, its powerful image is
	\begin{enumerate}
	\item $\kappa$-preaccessible in $\cl$, and
	\item closed in $\cl$ under $\theta^+$-small $\kappa$-directed colimits of $\kappa$-presentable objects.
	\end{enumerate}
\end{theo}

\begin{proof}
	We adhere closely to the argument for \cite[3.2]{BT-R} (although we consider the powerful, rather than the $\lambda$-pure powerful, image, and we argue via logical compactness rather than ultraproducts---see Remark~\ref{rmklpure}).  In fact, we will condense the first steps of the aforementioned argument, in which the powerful image of the functor $F$---say, for definiteness, that it is $\lambda$-accessible---is encoded first as the full image of the forgetful functor
$$H:\cm\to\cl$$
where $\cm$ is the category of monomorphisms $L\to FK$, $K\in\ck$; or, to be more precise, $\cm$ is the pullback of the comma category $(Id\downarrow F)$ in $\cl^{\to}$ along the inclusion $\Mono{\cl}\hookrightarrow \cl^{\to}$, $\cl^{\to}$ the usual arrow category of $\cl$.  From the pullback characterization, the category $\cm$ is $\mu_F$-accessible, and $H$ preserves $\lambda$-directed colimits.  

We now use the syntactic characterizations of accessible categories to simplify the problem further.  In short, the categories $\cl$ and $\cm$ can be fully embedded in categories of many-sorted finitary structures $\Str(\Sigma_\cl)$ and $\Str(\Sigma_\cm)$, respectively, and $\cm$ can be identified with a subcategory $\Mod(T)$ of $\Str(\Sigma_\cm)$ consisting of models of a $\Ll_{\mu_F,\mu_F}(\Sigma_\cm)$-theory $T$.  This can also be considered as a theory in signature $\Sigma=\Sigma_\cl\coprod\Sigma_\cm$, with $E:\Mod(T)\to\Str(\Sigma)$ the corresponding embedding.  Consider the reduct functor $R:\Mod(T)\stackrel{E}{\to}\Str(\Sigma)\to\Str(\Sigma_\cl)$.  The full image of $R$, which is precisely $\Red{\Sigma}{\Sigma_\cl}{T}$, the subcategory of reducts of models of $T$ to $\Sigma_\cl$, is equivalent to the full image of $H$.  

Notice that the functor $R$ preserves $\mu_F$-directed colimits and $\mu_F$-presentable objects.  Since $\kappa$ is, in particular, inaccessible, and since $\mu_F<\delta\leq\kappa$, it follows that $\mu_F\sless\kappa$. By design, then, $R$ preserves $\kappa$-directed colimits and $\kappa$-presentable objects.  This means, in particular, that every $M$ in $\Red{\Sigma}{\Sigma_\cl}{T}$ is a $\kappa$-directed colimit of $\kappa$-presentable structures in $\Str(\Sigma_\cl)$ that lie in $\Red{\Sigma}{\Sigma_\cl}{T}$.  That is, $\Red{\Sigma}{\Sigma_\cl}{T}$ is $\kappa$-preaccessible in $\Str(\Sigma_\cl)$, which gives us clause (1) of the theorem.  

As in \cite{BT-R} and \cite{lcatth}---and Fact~\ref{factlcsumm}(2)---the precise nature of the large cardinal comes into play only now, in ensuring the closure of $\Red{\Sigma}{\Sigma_\cl}{T}$ in $\Str(\Sigma_\cl)$ under colimits of diagrams of a particular size or shape---in our case, $\kappa$-directed diagrams of size at most $\theta$ and consisting of  $\kappa$-presentable objects.

To that end, let $D:I\to \Red{\Sigma}{\Sigma_\cl}{T}$ be such a diagram, and let $M$ be its colimit in $\Str(\Sigma_\cl)$.  To show that $M$ is in $\Red{\Sigma}{\Sigma_\cl}{T}$, it suffices to show that $M$ admits a monomorphism into an object of $\Red{\Sigma}{\Sigma_\cl}{T}$.  We may axiomatize monomorphisms out of $M$ by its atomic diagram; that is, the collection of all atomic and negated atomic formulas in $L_{\lambda\lambda}(\Sigma_M)$, where $\Sigma_M$ is $\Sigma_\cl$ with the addition of new constant symbols $c_m$ for each $m\in UM$.  So, in fact, it suffices to exhibit a model of the theory $T'=T\cup T_M$ in the new signature $\Sigma_M$---in terms of logical complexity, this theory lives in $\Ll_{\mu_F,\mu_F}(\Sigma_M)$, hence also in $\Ll_{\delta,\delta}(\Sigma_M)$.

Because it is a $\theta^+$-small, $\kappa$-directed colimit of $\kappa$-presentable objects, $M$ is $\theta^+$-presentable in $\Str(\Sigma_\cl)$ by \cite[2.3.11]{makkai-pare}.  Thus, $|UM|\leq\theta$.  So $|T'|\leq\theta$, and we are ideally positioned to make use of $(\delta,\theta)$-compactness of $\kappa$.

Let $\Gamma\subseteq T'$, $|\Gamma|<\kappa$.  In particular, $\Gamma\subseteq T\cup T_M^{\Gamma}$, where $T_M^{\Gamma}\subseteq T_M$ is of cardinality less than $\kappa$ and, by a simple counting argument, involves fewer than $\kappa$ of the new constants $c_m$.  As $\kappa$-directed colimits in $\Str(\Sigma_\cl)$ are concrete, $UM$ is the $\kappa$-directed union of the $UDi$, $i\in I$.  By $\kappa$-directedness, some particular $UDi$ will contain all of the relevant interpretations $c_m^M$ in $UM$ of the symbols in $T_M^{\Gamma}$.  Since $Di$ is in $\Red{\Sigma}{\Sigma_\cl}{T}$, and the identity (mono)morphism $id_{Di}:Di\to Di$ witnesses that $Di\models T_M^{\Gamma}$, we have that $Di\models \Gamma$.  

By $(\delta,\theta)$-compactness of $\kappa$, then, $T'$ has a model, and we are done.
\end{proof}

\begin{remark}\label{rmklpure} \emph{
	\begin{enumerate}
	\item Both \cite{BT-R} and \cite{lcatth} use ultrafilter characterizations of large cardinals, and carry out the final step of the argument with an ultraproduct construction.  
	\item In \cite{BT-R}, the authors are concerned with the $\lambda$-pure powerful image, which requires two minor modifications.  First, we define $\cm$ in a similar fashion, but insisting the morphisms be $\lambda$-pure. Second, when attempting to construct a $\lambda$-pure morphism from $M$ in the final step, we must use its \emph{$\lambda$-pure diagram}, consisting of positive primitive (and negated positive primitive) formulas, see \cite[p.~9]{BT-R}.
	\end{enumerate}
 }
\end{remark}

Second, we give the analogous result for compactness for ultrafilters.

\begin{theo}\label{thmcptclosednok}
	Let $\kappa$ be  $(\delta,\theta)$-compact for ultrafilters.  For any accessible functor $F:\ck\to\cl$ below $\delta$, $\cp$ is 
	\begin{enumerate}
	\item $\kappa$-preaccessible in $\cl$, and
	\item closed in $\cl$ under $\theta^+$-small $\kappa$-directed colimits.
	\end{enumerate}
\end{theo}
\begin{proof}
The proof proceeds almost exactly as in the proof of Theorem \ref{thmcptclosed}.  The only difference comes in the final two paragraphs (beginning with ``Because it is a $\theta^+$-small, $\kappa$-directed colimit of $\kappa$-presentable objects...''), as we now drop the assumption that the diagram consists of $\kappa$-presentable objects.  Here we assume only that $M$ is the colimit of a $\theta^+$-small, $\kappa$-directed diagram $D$ from $\Red{\Sigma}{\Sigma_\cl}{T}$.

Note that dropping the size restriction means that we no longer have any control over the size of the theory $T'$, which was crucial in the argument above.  Nonetheless, the diagram $D$ gives us a way to filtrate $T'$, indexed by $I = \dom D$.  In particular, given $i \in I$, set $T'_i:=T \cup T_{Di}$ in the signature $\Sigma_{Di}$ (viewed as a subsignature of $\Sigma_M$).  Then the satisfiability of $T'_i$ is witnessed by $Di$.  By the compactness for ultrafilters of $\kappa$ and Proposition \ref{categorified-prop}, $T'$ is satisfiable and we are done.
\end{proof}

%As mentioned, the use of cardinal arithmetic allows us to remove awkward hypotheses from the theorem starting with a $(\delta, \theta)$-strong compact using Remark \ref{c-u-sc-rem}.

%\begin{cor}\label{corthetaplus}
	%Let $\kappa$ be $(\delta,\theta)$-compact, $\theta^{<\kappa}=\theta$, and let $F:\mathcal{K}\to \cl$ be an accessible functor below $\delta$.  If $\pres{\kappa}{\cl}\subset \cp(F)$, then $\pres{\theta^+}{\cl}\subset \cp(F)$.
%\end{cor}
%\begin{proof}
% Since $\theta^{<\kappa}=\theta$, Fact~\ref{factsharp} implies that $\kappa\sless\theta^+$.  The category $\cl$ is $\kappa$-accessible ($\mu_F<\kappa$, and $\kappa$ is inaccessible, so $\mu_F\sless\kappa$), so by \cite[2.3.11]{makkai-pare}, any	$\theta^+$-presentable object is a $\theta^+$-small $\kappa$-directed colimit of $\kappa$-presentables.  The conclusion then follows from Theorem~\ref{thmcptclosed}.
%\end{proof}

\section{Powerful images and tameness}\label{sectame}

We now recast equivalence of Galois types in terms of powerful images, which will allow us to derive tameness from closure properties of the sort derived in Theorem~\ref{thmcptclosed}.  The approach is just as in \cite[5.2]{LRclass}, \cite[\S5]{LiRo17}, and \cite{lcatth}.  In particular, we work at the level of generality used in \cite[4.8]{lcatth}, focusing not on AECs but on \emph{concretely accessible categories}:

\begin{defin}\label{defconcracc}
	A \emph{concretely $\lambda$-accessible category} consists of a $\lambda$-accessible category $\ck$ and a faithful functor $U:\ck\to\Set$ that is $\lambda$-accessible and preserves $\lambda$-presentable objects.
\end{defin}

We begin by encoding equivalence of types in the form of the powerful image of a particular accessible functor---while we give the basic outline of the argument here, the details are the same as in the corresponding tameness and locality proofs mentioned above.\\

\begin{defin}\label{deftamecats} Let $\ck$ be a concretely $\lambda$-accessible category.
\begin{enumerate}
\item We define $\ck^<$ to be the category of pairs of (representatives of) types $(f_i,a_i)$, $i=0,1$, i.e. spans
			$$  
      \xymatrix{
        N_0  &  \\
        M \ar [u]^{f_0} \ar [r]_{f_1} &
        N_1 
      }
    $$
    with selected elements $a_i\in UN_i$, $i=0,1$.
    \item We denote by $\ck^\square$ the category of squares witnessing the equivalence of such pairs, i.e. commutative squares
    $$  
      \xymatrix{
        N_0 \ar[r]^{g_0} & N \\
        M \ar [u]^{f_0} \ar [r]_{f_1} &
        N_1 \ar[u]_{g_1}
      }
    $$
    with selected elements $a_i\in UN_i$, $i=0,1$, and $U(g_0)(a_0)=U(g_1)(a_1)$.  
    \item Let $F:\ck^\square\to\ck^<$ be the obvious forgetful functor.
    \end{enumerate}
    \end{defin}
    
    \begin{remark}\label{rmktamecats} \emph{
    \begin{enumerate}
    \item The full image of $F$, which is itself powerful---by inspection---consists of precisely the equivalent pairs $(f_1,a_1)$, $(f_2,a_2)$.
    \item By \cite[4.7]{lcatth}, $\ck^\square$ and $\ck^<$ are both $\lambda$-accessible.  Moreover, $F$ is $\lambda^+$-accessible, and preserves $\nu$-presentable objects for arbitrary $\nu\sgreat\lambda$.  
    \item\label{rmktamecatsbelow} Following the proof of  \cite[2.3]{lcatth}, we note that if $\lambda<\delta$, $\delta$ inaccessible, one can easily verify that $\mu_{\ck^<}=\mu_\ck<\delta$ and, in turn, that $\mu_F<\delta$.	
    \end{enumerate} }
    \end{remark}
    
    \begin{theo}\label{thmclosedtame}
    	Let $\delta\leq\kappa$ be inaccessible cardinals, and $\theta$ such that $\theta^+\sgeq\kappa$.  Suppose that every accessible functor $F$ below $\delta$ has powerful image closed under $\theta^+$-small $\kappa$-directed colimits.  Then any concretely $\lambda$-accessible category $\ck$, $\lambda<\delta$, is $(<\kappa,\theta)$-tame.
    \end{theo}

\begin{proof}
    
    Let $\ck$ be a concretely $\lambda$-accessible category, with $\lambda<\delta$, and consider $\ck^\square$, $\ck^<$, and $F:\ck^\square\to\ck^<$, as in Definition~\ref{deftamecats}. By Remark~\ref{rmktamecats}(\ref{rmktamecatsbelow}), $F$ is below $\delta$, so the closure hypothesis in the theorem holds for the powerful image of $F$: it is closed under $\theta^+$-small $\kappa$-directed colimits of $\kappa$-presentable objects.
    
    By inaccessibility of $\kappa$, and the fact that $\lambda<\mu_F<\delta\leq\kappa$, we must have $\lambda\sless\kappa$.  Thus $\ck$ is a concretely $\kappa$-accessible category.  Moreover, the assumption on $\theta$ ensures, following the proof of Corollary~\ref{corthetaplus}, that $\theta^+\sgreat\kappa$.  
    
    Let $M$ be a $\theta^+$-presentable object in $\ck$, and consider $(f_i,a_i)$, $i=1,2$, with $f_i:M\to N_i$, $a_i\in UN_i$.  Suppose that $(f_1\chi,a_1)$ and $(f_2\chi,a_2)$ are equivalent for any $\chi:X\to M$, $X$ $\kappa$-presentable.  By \cite[2.3.11]{makkai-pare}, $M$ is the colimit of a $\theta^+$-small $\kappa$-directed colimit of $\kappa$-presentable objects, say with cocone $(\phi_i:M_i\to M\,|\,i\in I)$, $|I|\leq\theta$.  So, in particular, $(f_1\phi_i,a_1)$ and $(f_2\phi_i,a_2)$ are equivalent for all $i\in I$.
    
    Hence the spans 
    $$  
      \xymatrix{
        N_1  &  \\
        M_i \ar [u]^{f_1\phi_i} \ar [r]_{f_2\phi_i} &
        N_2 
      }
    $$
    all belong to the powerful image of $F$.  As the original span
    $$  
      \xymatrix{
        N_1  &  \\
        M \ar [u]^{f_1} \ar [r]_{f_2} &
        N_2 
      }
    $$
    is clearly the $\kappa$-directed colimit of the $\theta^+$-small diagram of such spans, it belongs to the powerful image as well.  So $(f_1,a_1)$ and $(f_2,a_2)$ are equivalent, meaning that $\ck$ is $(<\kappa,\theta)$-tame.
    \end{proof}
    
\begin{remark}\label{rmkcpttamemods}\emph{\begin{enumerate}
	\item If we assume, as we do in the ensuing results of this paper, that $\theta^{<\kappa}=\theta$, then $\theta^+\sgeq\kappa$ and we are in the situation of the theorem.
	\item While we here assume closure under arbitrary $\theta^+$-small, $\kappa$-directed colimits, purely for the sake of simplicity, we only need closure under such colimits \emph{in which the objects are $\kappa$-presentable}.  The proof of this stronger result requires slightly more bookkeeping, as the spans in the powerful image produced above need not have $N_1$ and $N_2$ $\kappa$-presentable, meaning that the spans, as written, need not be $\kappa$-presentable, in which case the closure condition is of no help.  One can remedy this by representing $N_1$ and $N_2$ as $\theta^+$-small $\kappa$-directed diagrams of $\kappa$-presentable objects, and collating the individual diagrams into a diagram of spans of the appropriate kind: we omit the details.
	\item We note that this stronger assumption implies a separation principle for Galois types that may be of independent interest: two types over an object $M$ are equal just in case for any $\kappa$-directed, $\theta^+$-small diagram $D:I\to\ck$ with colimit $M$, the types in question have have the same restrictions to each $D(i)$.  This generalizes \emph{locality} for Galois types, which is defined in terms of \emph{chains}, rather than directed systems.  Although this notion seems to be a natural one (especially from a categorical viewpoint), it has not seen any use in the AEC literature and we will not develop it further here.\end{enumerate}}
\end{remark}
    
\begin{remark}\label{rmkaecstame} \emph{We note that any AEC $\ck$ is a concretely $\lsnk^+$-accessible category---and likewise for the more general $\mu$-AECs of \cite{mu-aec-jpaa}---so the hypotheses of Theorem~\ref{thmclosedtame} will imply $(<\kappa,\theta)$-tameness of all AECs with $\lsnk<\delta$.}	
\end{remark}

\section{Characterization Theorem}\label{secequiv}

We now build to the main result, Theorem~\ref{thmcptequiv}, which gives the promised equivalent category- and model-theoretic characterizations of $(\delta,\theta)$-compactness.  We begin by summarizing what we have proven so far:

\begin{thm}\label{thmcptimpl} 
Let $\delta\leq\kappa$ be inaccessible cardinals, and $\theta^{<\kappa}=\theta$.  Each of the following statements implies the next:
\begin{enumerate}
	\item $\kappa$ is $(\delta,\theta)$-compact.
	\item If $F:\ck\to\cl$ is an accessible functor below $\delta$, $\cp(F)$ is closed in $\cl$ under $\theta^+$-small $\kappa$-directed colimits.
	\item Every AEC $\bK$ with $\LS(\bK)<\delta$ is $(<\kappa, \theta)$-tame.
\end{enumerate}
\end{thm}

\begin{proof} 

$(1) \implies (2)$ follows from Remark~\ref{c-u-sc-rem} and Theorem~\ref{thmcptclosednok}.  The implication $(2)\implies (3)$ is simply \ref{thmclosedtame} (in conjunction with Remarks~\ref{rmkcpttamemods}(1) and \ref{rmkaecstame}). We note that $(1)\implies (3)$ was proven independently as \cite[4.5]{boneylarge}.

\end{proof}

Note that we do not have a perfect equivalence in Theorem~\ref{thmcptequiv}.  The following would close the loop, but the cardinal arithmetic is just off (since $\delta^{(\theta^{<\kappa})} \geq 2^\theta > \theta$).

\begin{fact}[{\cite[Theorem 4.9(3)]{boun}}]\label{bu-str-fact}
If every AEC $\bK$ with $\LS(\bK) =\delta$ is $(<\kappa, \delta^{(\theta^{<\kappa})})$-tame, then $\kappa$ is $(\delta^+, \theta)$-compact for ultrafilters.
\end{fact}

At the very least, this allows us to give a nice characterization at $\kappa$-closed, strong limit cardinals.  Adopting the notational convention that a class is $(\mu, <\chi)$-tame if it is $(\mu, \chi_0)$-tame for all $\mu \leq \chi_0<\chi$, and that a diagram is $<\theta$-small if it is $\nu$-small for some $\nu<\theta$, we have:

\begin{thm}\label{thmcptequiv} 
Let $\delta\leq\kappa$ be inaccessible cardinals, and $\theta$ be a $\kappa$-closed strong limit cardinal. The following are equivalent:
\begin{enumerate}
	\item $\kappa$ is $(\delta,<\theta)$-compact.
	\item If $F:\ck\to\cl$ is $\lambda$-accessible, $\mu_\cl<\delta$, and preserves $\mu_\cl$-presentable objects, then the inclusion $\cp(F)\hookrightarrow\cl$ is strongly $\kappa$-accessible and $\cp(F)$ is closed under $<\theta$-small $\kappa$-directed colimits.
	\item Any AEC $\bK$ with $\lsnk<\delta$ is $(<\kappa, <\theta)$-tame.
\end{enumerate}
\end{thm}

\begin{proof}
Combine Theorem \ref{thmcptimpl} with Fact \ref{bu-str-fact} (recalling, again, Remark~\ref{c-u-sc-rem}).
\end{proof}

One final time, we note that this argument can be adapted, in straightforward fashion, to the case of $\cp_\lambda(F)$, the $\lambda$-pure powerful image, via the modifications suggested in Remark~\ref{rmklpure}(2).

\section{Further considerations}

The essential drift of this work is that large cardinals in the range from weakly to strongly compact admit both model- and category-theoretic characterizations, involving, respectively, tameness of AECs and closure properties of the powerful images of accessible functors.  Moreover, these equivalences pass to parametrizations of these notions that are intrinsically natural to each of the three disciplines.  

While it seems unlikely that the methods of this paper can be stretched to obtain similar model- and category-theoretic characterizations of other large cardinals (e.g. Woodin, supercompact, Vop\v enka, $C^{(n)}$-extendible), our excellent referee has asked if, in particular, one can find \emph{other} properties of accessible functors that \emph{do} characterize them---this question has already been addressed on the model-theoretic side (cf. \cite{boneymodlcs,bdgm}).  Per \cite[Ch. 6]{adamek-rosicky}, of course, Vop\v enka's Principle admits a remarkably nice equivalent: every subfunctor of an accessible functor is accessible.  We are largely content to leave the pursuit of further results along these lines to future work.

It is worth noting, though, that one of the innovations of the proof of \cite[3.2]{BT-R} (which underlies our Theorem~\ref{thmcptclosed}) is to provide a version of the argument using the characterization of compact cardinals in terms of elementary embeddings of the set-theoretic universe.  It is the opinion of the second author that this may offer a way to generate category-theoretic consequences---if not characterizations---of large cardinals of the types mentioned above, by simply testing the argument with the corresponding elementary embedding properties.

%The general theme of characterizing a wide array of large cardinals by model-theoretic compactness is explored in \cite{boneymodlcs, bdgm}.  However, in moving beyond the cardinals discussed here, at least one of two phenomena occur: either type-omission is a crucial part of the compactness (building on a result of Benda \cite{benda} for supercompacts) or the compactness involves a logic stronger than (or orthogonal to) the infinitary logics $\bL_{\kappa, \lambda}$.

%In connection with the first difficulty, we submit that tameness of Galois types is a generalization of a consequence of the first-order compactness theorem to AECs.  On the other hand, type-omitting compactness (which can characterize supercompact and huge cardinals in $\bL_{\kappa,\lambda}$) is a phenomenon that fails in first-order logic.  Thus we expect a different approach would have to be taken to obtain an AEC characterization of supercompacts.

%As for the second, Shelah's Presentation Theorem \cite{shelahaecs} reveals that any AEC can be seen as a reduct of a $\bL_{\kappa, \omega}$-axiomatizable class.  This means that tameness (or AEC properties more generally) are unlikely to have such a tight connection with stronger logics as we've seen here with the infinitary logics.  Thus, any AEC characterization of, e.g., strong or Woodin cardinals will have to be more subtle.  The same can be said of $C^{(n)}$-extendible cardinals and, therefore, Vopenka's Principle.  As accessible categories are themselves axiomatizable in infinitary logic, it seems likely that these difficulties will extend to them as well.

\bibliographystyle{alpha}
\bibliography{compactaccessible}

\end{document}